\documentclass{aims}
\usepackage{amsmath, amssymb, latexsym, enumerate, mathrsfs}
  \usepackage{paralist}
  \usepackage{graphics} 
  \usepackage{epsfig} 
\usepackage{graphicx}  \usepackage{epstopdf}
 \usepackage[colorlinks=true]{hyperref}
\hypersetup{urlcolor=blue, citecolor=red}

\def\currentvolume{X}
 \def\currentissue{X}
  \def\currentyear{200X}
   \def\currentmonth{XX}
    \def\ppages{X--XX}
     \def\DOI{10.3934/xx.xx.xx.xx}

\newtheorem{theorem}{Theorem}[section]

\newtheorem{lemma}[theorem]{Lemma}
\newtheorem{proposition}[theorem]{Proposition}

\newtheorem{problem}{Problem}
\newtheorem{definition}[theorem]{Definition}
\newtheorem{assumption}[theorem]{Assumption}
\newtheorem{remark}[theorem]{Remark}

\title[On the attainable distributions of controlled-diffusion processes] 
      {On the attainable distributions of controlled-diffusion processes pertaining to a chain of distributed systems}

\author[Getachew K. Befekadu and Eduardo L. Pasiliao]{}


 \email{gbefekadu@ufl.edu}
 \email{pasiliao@eglin.af.mil}

\thanks{This research was supported in part by the Air Force Research Laboratory (AFRL). This work is, in some sense, a continuation of our previous paper \cite{BefP15}}

\thanks{$^*$ Corresponding author: Getachew K. Befekadu}

\begin{document}
\maketitle

\centerline{\scshape Getachew K. Befekadu}
\medskip
{\footnotesize
 \centerline{Department of Mechanical and Aerospace Engineering} 
   \centerline{University of Florida - REEF}
   \centerline{1350 N. Poquito Rd, Shalimar, FL 32579, USA}
} 

\medskip

\centerline{\scshape Eduardo L. Pasiliao}
\medskip
{\footnotesize
 \centerline{ Munitions Directorate, Air Force Research Laboratory}
   \centerline{101 West Eglin Boulevard,}
   \centerline{Eglin AFB, FL 32542, USA}
}

\bigskip

 \centerline{(Communicated by the associate editor name)}

\begin{abstract}
We consider a controlled-diffusion process pertaining to a chain of distributed systems with random perturbations that satisfies a weak H\"{o}rmander type condition. In particular, we consider a stochastic control problem with the following objectives that we would like to achieve. The first one being of a reachability-type that consists of determining a set of attainable distributions at a given time starting from an initial distribution; while the second one involves minimizing the relative entropy subject to the initial and desired final attainable distributions. Using the logarithmic transformations approach from Fleming, we provide a sufficient condition on the existence of an optimal admissible control for such a stochastic control problem which is amounted to changing the drift by a certain perturbation suggested by Jamison in the context of reciprocal processes. Moreover, such a perturbation coincides with a minimum energy control among all admissible controls forcing the controlled-diffusion process to the desired final attainable distribution starting from the initial distribution. Finally, we briefly remark  on the invariance property of the path-space measure for such a diffusion process pertaining to the chain of distributed systems.
\end{abstract}

\section{Introduction}
The problem of forcing a diffusion process to a final attainable configuration starting from an initial distribution has been of particular interest from both physical and mathematical points of view (e.g., see \cite{Sch31} for the original formulation of this problem). Notably, this problem and its variants have been widely studied in the literature (e.g., see \cite{Leo12}, \cite{Leo14} \cite{MikT06} or \cite{Mik15} in the context of stochastic optimal transportation problem;  see also \cite{Nag89} and \cite{Wak89} in the context of variational characterization of Schr\"{o}dinger processes; and \cite{FleS85} or \cite{Dai91} in the context of controllability of systems governed by parabolic PDEs). The setting of these papers, which is the rationale behind our approach, consists of constructing a class of stochastic processes (i.e., a reciprocal precess along the works of \cite{Jam74}, \cite{Jam75} and \cite{Beu60}) in connection with stochastic optimal control theory with respect to the desired end-point distributions or in connection with that of the problem of minimizing relative entropy subject to the controlled and reference diffusion processes (e.g., see \cite{Leo14}, \cite{MikT06} or \cite{Dai91} for additional discussions). Here, our main interest is to throw some light on the structure of the controlled-diffusion process pertaining to a chain of distributed systems and, at the same time, clarifying questions concerning minimization of relative entropy subject to the initial and desired final attainable distributions for such a controlled-diffusion process.\footnote{In this paper, our intent is to provide a theoretical framework, rather than considering a specific numerical problem or application.}

In this paper, we specifically consider the following distributed system, which is formed by a chain of $n$ subsystems (where $n \ge 2$), with a random perturbation that enters only in the first subsystem and is then subsequently transmitted to other subsystems, i.e.,
\begin{eqnarray}
\left.\begin{array}{l}
d x_t^1 = m_1\bigl(t, x_t^1, \ldots,  x_t^n\bigr) dt + \sigma\bigl(t, x_t^1, \ldots, x_t^n\bigr)dW_t \\
d x_t^2 = m_2\bigl(t, x_t^1, \ldots,  x_t^n\bigr) dt  \\
d x_t^3 = m_3\bigl(t, x_t^2, \ldots,  x_t^n\bigr) dt  \\
 \quad\quad\quad~~ \vdots  \\
d x_t^n = m_n\bigl(t, x_t^{n-1}, x_t^n\bigr) dt, \quad 0 \le t \le T
\end{array}\right\},  \label{Eq1} 
\end{eqnarray}
where
\begin{itemize}
\item $x^i$ is an $\mathbb{R}^{d}$-valued state information for the $i$th subsystem, with $i \in \{1, 2, \ldots, n\}$, 
\item $m_1 \colon \mathbb{R}_{+} \times \mathbb{R}^{nd} \rightarrow \mathbb{R}^d$ and $m_{j} \colon \mathbb{R}_{+} \times \mathbb{R}^{(n-j+2)d} \rightarrow \mathbb{R}^d$, for $j = 2, \ldots, n$, are bounded continuous functions and satisfy appropriate H\"{o}lder conditions,
\item $\sigma \colon \mathbb{R}_{+} \times \mathbb{R}^{nd} \rightarrow \mathbb{R}^{d \times d}$ is a bounded continuous function, with the least eigenvalue of $\sigma\,\sigma^T$ uniformly bounded away from zero, i.e., 
\begin{align*}
 \sigma\bigl(t, x^1, \ldots, x^n\bigr)\,\sigma^T\bigl(t, x^1, \ldots, x^n\bigr) \succeq \lambda I_{d \times d}, \quad \forall (x^1, \ldots, x^n) \in \mathbb{R}^{nd}, \quad \forall t \in \mathbb{R}_{+}
\end{align*}
for some $\lambda > 0$,
\item $W_t$ (with $W_0 = 0$) is a $d$-dimensional standard Wiener process.
\end{itemize}
Note that such a chain of distributed systems has been well discussed in various applications (e.g., see \cite{BodL08}, \cite{BefA15}, \cite{DelM10} and \cite{Soi94} and the references therein). For example, when $n=2$, the equation in \eqref{Eq1} can be used to describe stochastic Hamiltonian systems (e.g., see \cite{BodL08} or \cite{Soi94} for additional discussions). 

Next, let us introduce the following notation that will be useful later. We use bold face letters to denote variables in $\mathbb{R}^{nd}$, for instance, $\mathbf{0}$ stands for a zero in $\mathbb{R}^{nd}$ (i.e., $\mathbf{0} \in \mathbb{R}^{nd}$) and, for any $t \ge 0$, the solution $\bigl(x_t^1, x_t^2, \ldots,  x_t^n\bigr)$ to \eqref{Eq1} is denoted by $\mathbf{x}_t$. Moreover, for $\bigl(t, (x^{j-1}, \ldots, x^n)\bigr) \in \mathbb{R}_{+} \times \mathbb{R}^{(n-j+2)d}$, $j = 2, \ldots, n$, the function $x^j \mapsto m_j \bigl(t, x^{j-1}, \ldots, x^n\bigr)$ is continuously differentiable with respect to $x^j$ and its derivative denoted by $\bigl(t, x^{j-1}, \ldots, x^n\bigr) \mapsto D_{x^j} m_j \bigl(t, x^{j-1}, \ldots, x^n\bigr)$.

Then, we can rewrite the stochastic differential equation (SDE) in \eqref{Eq1} as follow
\begin{align}
  d \mathbf{x}_t = \mathbf{M} (t, \mathbf{x}_t ) dt + G \sigma (t, \mathbf{x}_t ) dW_t,  \label{Eq2}
\end{align}
where $\mathbf{M} = \bigl [m_1, m_2, \ldots, m_n\bigr ]$ is an $\mathbb{R}^{nd}$-valued function and $G = \bigl[ I_d, 0, \ldots, 0 \bigr ]^T$ stands for an $(nd \times d)$ matrix that embeds $\mathbb{R}^d$ into $\mathbb{R}^{nd}$. Moreover, the infinitesimal generator associated with \eqref{Eq2} is given by\footnote{$\mathbf{x}^{j-1} \triangleq (x^{j-1}, \ldots, x^n)$ for $j = 2, \ldots n$.}
\begin{align}
 \mathcal{L}_{t,\mathbf{x}} = \dfrac{1}{2} \operatorname{tr} \bigl(a(t, \mathbf{x}) D_{x^1}^2\bigr) + m_1(t, \mathbf{x}) \cdot D_{x^1} + \sum\nolimits_{j=2}^n m_j(t, \mathbf{x}^{j-1}) \cdot  D_{x^j},  \label{Eq3}
\end{align}
where $a(t, \mathbf{x})=\sigma(t, \mathbf{x}) \sigma^T(t, \mathbf{x})$.

\begin{remark} \label{R1}
Note that, in \eqref{Eq1}, the random perturbation enters only in the first subsystem through its diffusion and is then subsequently transmitted to other subsystems through their respective drift terms. As a result, such a chain of distributed systems is described by an $\mathbb{R}^{nd}$-valued diffusion process $\mathbf{x}_t$, which is degenerate in the sense that the second-order operator associated with it is a degenerate parabolic equation. Moreover, we assume that the distributed system in \eqref{Eq1} satisfies a weak H\"{o}rmander type condition (e.g., see \cite{Hor67} or \cite[Section~3]{Ell73} for additional discussions).
\end{remark}

Throughout this paper, we assume that the following statements hold for the distributed system in \eqref{Eq2} (or \eqref{Eq1}).

\begin{assumption} \label{AS1} ~\\\vspace{-3mm}
\begin{enumerate} [(a)]
\item The functions $m_1(t, \mathbf{x})$ and $m_{j}(t, \mathbf{x}^{j-1})$ for $j = 2, \ldots, n$ satisfy appropriate H\"{o}lder conditions with respect to $\mathbf{x}$ and $\mathbf{x}^{j-1}$, respectively. Moreover, $a(t, \mathbf{x})$ is a bounded $C^{2} \bigl([0, T] \times \mathbb{R}^{nd}\bigr)$-function;  $a(t, \mathbf{x})$ and $D_{x^i}a(t, \mathbf{x})$ are bounded and satisfy appropriate H\"{o}lder conditions with respect to both $\mathbf{x}$ and $t$.
\item The infinitesimal generator $\mathcal{L}_{t,\mathbf{x}}$ is hypoelliptic (e.g., see \cite{Hor67} or \cite{Ell73}).
\end{enumerate}
\end{assumption}

\begin{remark} \label{R2}
In general, the hypoellipticity assumption is related to a strong accessibility property of controllable nonlinear systems that are driven by white noise (e.g., see \cite{SusJu72} concerning the controllability of nonlinear systems, which is closely related to \cite{StrVa72}; see also \cite[Section~3]{Ell73}). Moreover, the Jacobian matrices $D_{x^1} m_1(t, \mathbf{x})$ and $D_{x^{j-1}} m_j(t, \mathbf{x}^{j-1})$ for $j=2, \ldots, n$ are assumed to be nondegenerate, uniformly in time and space. Note that the hypoellipticity assumption also implies that the diffusion process $\mathbf{x}_t$ has a transition density with a strong Feller property.
\end{remark}

\begin{remark} \label{R3}
Here, it is worth mentioning that there are some results, based on Malliavin calculus under strong H\"{o}rmander conditions, that provide a sensitivity measure of the system with respect to noise, where a nonzero sensitivity condition (which is summarized by a nondegenerate {\it Malliavin matrix}) is used in verifying the existence of such a density function (e.g., see \cite{KusS87} and \cite{Nua95} for additional discussions).  
\end{remark}

The paper is organized as follows. In Section~\ref{S2}, we provide some preliminary results. Section~\ref{S3} formally states the stochastic control problem considered in this paper. In Section~\ref{S4}, using the logarithmic transformations approach from Fleming, we provide a sufficient condition on the existence of an optimal admissible control for such a stochastic control problem. This section also contains additional results concerning minimization of relative entropy subject to controlled and reference diffusion processes. Finally, in Section~\ref{S5}, we briefly remark on the invariance property of the path space measure for such a diffusion process pertaining to the chain of distributed systems.

\section{Preliminaries} \label{S2} 
In this section, we provide some preliminary results that will be useful later in Section~\ref{S4}. Note that, for any square integrable $\mathbb{R}^{nd}$-valued random variable $\mathbf{\xi}$ that is independent to $\bigl\{W_t; \, 0 \le t \le T \bigr\}$, the SDE in \eqref{Eq2} admits a weak solution in $[0, T)$ (i.e., in $[0, T - \varepsilon]$ for any $\varepsilon > 0$) with initial condition $\mathbf{x}_0=\mathbf{\xi}$. 

Moreover, the fundamental solution (i.e., the transition density) $q(s, \mathbf{x}, t, \mathbf{y})$ for $0 \le s < t$ and $\mathbf{x}, \mathbf{y} \in  \mathbb{R}^{nd}$ of the PDE of parabolic type
satisfies the following 
\begin{align}
 \frac{\partial}{\partial t} q(t, \mathbf{x}, T, \mathbf{y}) + \mathcal{L}_{t,\mathbf{x}} q(t, \mathbf{x}, T, \mathbf{y}) &= 0 \quad \text{in} \quad [0,T) \times \mathbb{R}^{nd} \notag \\
   \lim_{t \uparrow T} q(t, \mathbf{x}, T, \mathbf{y}) &= \delta_{\mathbf{y}}(\mathbf{x}) \quad \text{for} \quad \mathbf{x}, \mathbf{y} \in \mathbb{R}^{nd} \label{Eq4}
\end{align}
and it is twice continuously differentiable with respect to $\mathbf{x}$ and continuously differentiable with respect to $s$. Note that, for a fixed arriving point $(T, \mathbf{\zeta}) \in [0, \infty) \times \mathbb{R}^{nd}$, we can approximate the boundary condition in \eqref{Eq4} using a sequence of positive functions $(\phi_{\varepsilon})_{\varepsilon > 0}$ on $\mathbb{R}^{nd}$ that weakly converge towards the Dirac function $\delta_{\mathbf{\zeta}}$. 

To this end, we assume that $(\phi_{\varepsilon})_{\varepsilon > 0}$ (on the whole $\mathbb{R}^{nd}$) satisfies the following
\begin{align}
\exists_{\varepsilon_0 > 0} \quad \text{such that} \quad \lim_{c \rightarrow \infty}  \sup_{0 < \varepsilon < \varepsilon_0}  \sup_{\vert \mathbf{x} \vert > c} \phi_{\varepsilon}(\mathbf{x}) = 0.  \label{Eq5} 
\end{align}
Then, we can approximate the transition density function by
\begin{align}
h_{\varepsilon}(t, \mathbf{x}) = \mathbb{E}_{t, \mathbf{x}} \bigl\{\phi_{\varepsilon}(\mathbf{\zeta})\bigr\}, \quad \forall  \varepsilon > 0, \quad \forall (t, \mathbf{x}) \in [0, T-\varepsilon] \times \mathbb{R}^{nd},  \label{Eq6}
\end{align}
where such an approximation also satisfies the following Cauchy problem
\begin{align}
 \frac{\partial}{\partial t} h_{\varepsilon}(t, \mathbf{x}) + \mathcal{L}_{t,\mathbf{x}} h_{\varepsilon}(t, \mathbf{x}) = 0 \quad \text{in} \,\, [0,T- \varepsilon] \times \mathbb{R}^{nd},  \label{Eq7}
\end{align}
with boundary condition $h_{\varepsilon}(T - \varepsilon, \mathbf{x}) = \phi_{\varepsilon}(\mathbf{x})$ for $\mathbf{x} \in \mathbb{R}^{nd}$. Note that, since $q$ is continuous, we have the following
\begin{align}
\lim_{\varepsilon \rightarrow 0} h_{\varepsilon}(0, \mathbf{x}) &= \lim_{\varepsilon \rightarrow 0} \mathbb{E}_{0, \mathbf{x}} \bigl\{\phi_{\varepsilon}(\mathbf{\zeta})\bigr\} \notag \\
&= \lim_{\varepsilon \rightarrow 0} \int_{\mathbb{R}^{nd}} \phi_{\varepsilon}(\mathbf{y}) q(0, \mathbf{x}, T- \varepsilon,\mathbf{y}) d \mathbf{y} \notag \\
&= q(0, \mathbf{x}, T, \mathbf{\zeta}).  \label{Eq8}
\end{align}
If we introduce the following logarithmic transformation (e.g., see Fleming \cite{Flem78} or \cite{Flem82} for such transformations in the context of stochastic control arguments) 
\begin{align}
J_{\varepsilon}(t, \mathbf{x}) = - \log h_{\varepsilon}(t, \mathbf{x}), \quad (t, \mathbf{x}) \in [0, T-\varepsilon] \times \mathbb{R}^{nd},  \label{Eq9}
\end{align}
then it is easy to show that $J_{\varepsilon}(t, \mathbf{x})$ satisfies the following nonlinear parabolic equation
\begin{align}
 \frac{\partial}{\partial t} J_{\varepsilon}(t, \mathbf{x}) + \mathcal{L}_{t,\mathbf{x}} J_{\varepsilon}(t, \mathbf{x}) = - \frac{1}{2} a(t, \mathbf{x}) D_{x^1} J_{\varepsilon}(t, \mathbf{x}) \cdot D_{x^1} J_{\varepsilon}(t, \mathbf{x})  \label{Eq10}
\end{align}
in $[0,T-\varepsilon] \times \mathbb{R}^{nd}$ with the following boundary condition
\begin{align}
J_{\varepsilon}(T - \varepsilon, \mathbf{x}) = - \log \phi_{\varepsilon}(\mathbf{x}), \quad \mathbf{x} \in \mathbb{R}^{nd}.  \label{Eq11}
\end{align}
Furthermore, $J_{\varepsilon}(t, \mathbf{x})$ is a value function to the following stochastic control problem\footnote{$\bigl \Vert u_t \bigr \Vert_{a^{-1}}^2 \triangleq \bigl \Vert \sigma^{-1}(t, \mathbf{x}_t) u_t \bigr \Vert^2$.}
\begin{align}
\inf_{u_{\cdot}} \mathbb{E}_{t, \mathbf{x}}  \left \{ \int_{t}^{T - \varepsilon} \dfrac{1}{2}\bigl \Vert u_t \bigr \Vert_{a^{-1}}^2 dt  - \log \phi_{\varepsilon} (\mathbf{x}_{T- \varepsilon}) \right \}  \label{Eq12}
\end{align}
subject to a controlled version of \eqref{Eq2}, i.e., 
\begin{align}
d \mathbf{x}_t^{u} = \Bigl( \mathbf{M} (t, \mathbf{x}_t^{u} ) + G u_t \Bigr) dt + G \sigma (t, \mathbf{x}_t^{u}) dW_t, \quad  \mathbf{x}_0^{u} = \mathbf{x},  \label{Eq13}
\end{align}
where $(u_t)_{0 \le t \le T-\varepsilon}$ is an $\mathbb{R}^{d}$-valued progressively measurable process satisfying 
\begin{align*}
 \mathbb{E} \int_0^{T - \varepsilon} \dfrac{1}{2}\bigl \Vert u_t \bigr \Vert_{a^{-1}}^2 dt < +\infty. 
\end{align*}
Note that, for a given $(t, \mathbf{x}) \in [0, T-\varepsilon] \times \mathbb{R}^{nd}$, the infimum in \eqref{Eq12} is achieved when
\begin{align}
u^{\ast}(t, \mathbf{x}) = a(t, \mathbf{x}) D_{x^1} \log h_{\varepsilon}(t, \mathbf{x}).  \label{Eq14}
\end{align}
Later, in Sections~\ref{S3} and \ref{S4}, we consider a stochastic control problem, where the objective is to force the controlled-diffusion process $\mathbf{x}_t^{u}$ to the desired final attainable distribution starting from the initial distribution using minimum energy control.

In what follows, let $g(\mathbf{x})$ be any positive measurable function that satisfies
\begin{align}
\int_{\mathbb{R}^{nd}} q(0, \mathbf{x}, T, \mathbf{z}) g(\mathbf{z}) d \mathbf{z} < +\infty \,\,\, \text{for some} \,\,\, \mathbf{x} \in \mathbb{R}^{nd}.  \label{Eq15}
\end{align}
Then, the function
\begin{align}
h(t, \mathbf{x}) = \int_{\mathbb{R}^{nd}} q(t, \mathbf{x}, T, \mathbf{z}) g(\mathbf{z}) d \mathbf{z}  \label{Eq16}
\end{align}
belongs to $C_b^{1, 2} \bigl([0, T] \times \mathbb{R}^{nd}\bigr)$ and it is also the kernel of the operator $\bigl({\partial}/{\partial t} + \mathcal{L}_{t,\mathbf{x}}\bigr)$, i.e, ${\partial} h(t, \mathbf{x})/{\partial t} + \mathcal{L}_{t,\mathbf{x}} h(t, \mathbf{x}) = 0$ in $[0, T) \times \mathbb{R}^{nd}$.

Note that an absolutely continuous change of measure on the path-space is related to changing the original drift term of the diffusion process associated with the SDE in \eqref{Eq2} (see also \cite{Gir60}). A particular case was considered in \cite{Jam75} (cf. \cite[Theorem~2]{Jam75}) leading to the following result.

\begin{proposition} \label{P1}
Suppose that $\mathbf{x}_t$ is a weak solution of \eqref{Eq2} in $[0, T)$. Let $h(t, \mathbf{x}) \in C_b^{1, 2} \bigl([0, T] \times \mathbb{R}^{nd}\bigr)$  be a strictly positive solution to following
\begin{align}
  \frac{\partial}{\partial t} h_{\varepsilon}(t, \mathbf{x}) + \mathcal{L}_{t,\mathbf{x}} h_{\varepsilon}(t, \mathbf{x}) = 0 \quad \text{in} \quad [0, T) \times  \mathbb{R}^{nd}   \label{Eq17}
\end{align}
such that $\mathbb{E} \bigl\{ h(t, \mathbf{x}) \bigr\} < +\infty$ and $h(s, \mathbf{x}) = \mathbb{E}_{s, \mathbf{x}} \bigl\{ h(t, \mathbf{x}) \bigr\}$ for all $0 \le s < t < T$. Then, the following SDE
\begin{align}
  d \mathbf{x}_t^{h} = \Bigl( \mathbf{M} (t, \mathbf{x}_t^{h} ) + G\, a(t, \mathbf{x}_t^{h}) D_{x^1} \log h(t, \mathbf{x}_t^{h}) \Bigr) dt + G \sigma (t, \mathbf{x}_t^{h}) dW_t  \label{Eq18}
\end{align}
admits a weak solution in $[0, T)$. Moreover, if there exists a positive measurable function $g(\mathbf{x})$ such that
\begin{align}
h(s, \mathbf{x}) = \mathbb{E}_{s, \mathbf{x}} \bigl\{ g(\mathbf{x}_T) \bigr\},  \label{Eq19}
\end{align}
then the transition density corresponding to \eqref{Eq18} is given by
\begin{align}
q^h(s, \mathbf{x}, t, \mathbf{y}) = q(s, \mathbf{x}, t, \mathbf{y})\dfrac{h(t, \mathbf{y})} {h(s, \mathbf{x})}.  \label{Eq20}
\end{align}
\end{proposition}

\begin{proof}
Let $(\Omega, P, \mathscr{F})$ be the probability space in which the weak solution $\mathbf{x}$ (i.e., a continuous process) of \eqref{Eq2}, with initial condition $\mathbf{x}_0$, is defined on. Then, let us introduce the following nonnegative martingale process 
\begin{align}
z_{t} = \dfrac{h(t, \mathbf{x}_t)}{h(0, \mathbf{x}_0)}  \label{Eq21}
\end{align}
with respect to the natural filtration $\mathscr{F}_t = \sigma\bigl\{\mathbf{x}_s \,\vert \, 0 \le s \le t \bigr\}$. 

Note that $\mathbb{E} \bigl\{ z_{t} \bigr\}=1$, then we can introduce the following change of probability measures on $\Omega$
\begin{align}
\dfrac{dQ}{dP} = z_{T-\varepsilon}  \label{Eq22}
\end{align}
for any fixed $\varepsilon > 0$.

Let $f \in C_0^{\infty}(\mathbb{R}^{nd})$ and $0 \le s \le t \le T-\varepsilon$, we have
\begin{align}
&\mathbb{E}_{s,\mathbf{x}}^{Q} \bigl\{f(\mathbf{x}_t)\bigr\} - f(\mathbf{x}_s) \notag \\
& = \dfrac{1}{h(s, \mathbf{x}_s)} \mathbb{E}_{s,\mathbf{x}} \Bigl \{ h(t, \mathbf{x}_t)f(\mathbf{x}_t) - h(s, \mathbf{x}_s)f(\mathbf{x}_s) \Bigr \} \notag \\
& = \dfrac{1}{h(s, \mathbf{x}_s)} \mathbb{E}_{s,\mathbf{x}} \Bigl \{ \int_s^t \Bigr [h(\tau, \mathbf{x}_{\tau})\mathcal{L}_{\tau,\mathbf{x}} f(\mathbf{x}_{\tau}) + a(\tau, \mathbf{x}_{\tau}) D_{x^1} \log h(\tau, \mathbf{x}_{\tau}) \cdot D_{x^1}f(\mathbf{x}_{\tau}) \Bigr]d\tau \Bigr \} \notag \\
& = \dfrac{1}{h(s, \mathbf{x}_s)} \mathbb{E}_{s,\mathbf{x}} \Bigl \{ \int_s^t h(\tau, \mathbf{x}_{\tau}) \Bigr [m_1(\tau, \mathbf{x}_{\tau}) \cdot D_{x^1} f(\mathbf{x}_{\tau}) + \sum\nolimits_{j=2}^n m_j(\tau, \mathbf{x}_{\tau}^{j-1}) \cdot D_{x^j} f(\mathbf{x}_{\tau}) \notag \\
&  \quad\quad\quad + \dfrac{1}{2} \operatorname{tr} \bigl(a(\tau, \mathbf{x}_{\tau}) D_{x^1}^2 f(\mathbf{x}_{\tau})\bigr) + a(\tau, \mathbf{x}_{\tau}) D_{x^1} \log h(\tau, \mathbf{x}_{\tau}) \cdot D_{x^1}f(\mathbf{x}_{\tau})\Bigr]d\tau \Bigr \} \notag \\
& = \dfrac{1}{h(s, \mathbf{x}_s)} \mathbb{E}_{s,\mathbf{x}} \Bigl \{ h(t, \mathbf{x}_{t}) \int_s^t \Bigr [m_1(\tau, \mathbf{x}_{\tau}) \cdot D_{x^1} f(\mathbf{x}_{\tau}) + \sum\nolimits_{j=2}^n m_j(\tau, \mathbf{x}_{\tau}^{j-1}) \cdot D_{x^j} f(\mathbf{x}_{\tau}) \notag \\
&  \quad\quad\quad + \dfrac{1}{2} \operatorname{tr} \bigl(a(\tau, \mathbf{x}_{\tau}) D_{x^1}^2 f(\mathbf{x}_{\tau})\bigr) + a(\tau, \mathbf{x}_{\tau}) D_{x^1} \log h(\tau, \mathbf{x}_{\tau}) \cdot D_{x^1}f(\mathbf{x}_{\tau})\Bigr] d\tau \Bigr \} \notag \\
& =  \mathbb{E}_{s,\mathbf{x}}^{Q}  \Bigl \{ \int_s^t \Bigr [m_1(\tau, \mathbf{x}_{\tau}) \cdot D_{x^1} f(\mathbf{x}_{\tau}) + \sum\nolimits_{j=2}^n m_j(\tau, \mathbf{x}_{\tau}^{j-1}) \cdot D_{x^j} f(\mathbf{x}_{\tau}) \notag \\
&  \quad\quad\quad + \dfrac{1}{2} \operatorname{tr} \bigl(a(\tau, \mathbf{x}_{\tau}) D_{x^1}^2 f(\mathbf{x}_{\tau})\bigr) + a(\tau, \mathbf{x}_{\tau}) D_{x^1} \log h(\tau, \mathbf{x}_{\tau}) \cdot D_{x^1}f(\mathbf{x}_{\tau})\Bigr] d\tau \Bigr \},    \label{Eq23}
\end{align}
where we have employed the It\^{o}'s rule for $h(t, \mathbf{x}_t)f(\mathbf{x}_t)$. This means that the law of $\mathbf{x}_{\cdot}$, as a process defined in $(\Omega, Q, \mathscr{F})$, solves the martingale problem for
\begin{align}
 \mathcal{L}_{t,\mathbf{x}}^{h} = m_1(t, \mathbf{x}_t) \cdot D_{x^1} + \sum\nolimits_{j=2}^n m_j(t, \mathbf{x}_t^{j-1}) \cdot D_{x^j} +  \dfrac{1}{2} \operatorname{tr} \bigl(a(t, \mathbf{x}_{\tau}) D_{x^1}^2\bigr) \notag \\
 + a(t, \mathbf{x}_t) D_{x^1} \log h(t, \mathbf{x}_t) \cdot D_{x^1}   \label{Eq24}
\end{align}
in $[0, T-\varepsilon]$. This is equivalent to saying that \eqref{Eq18} has a weak solution $\mathbf{x}_t^{h}$ in $[0, T-\varepsilon]$.

Note that $f$ has compact support and if $h(t, \mathbf{x}_t)=\mathbb{E}_{s,\mathbf{x}}\bigl\{g(\mathbf{x}_T)\bigr\}$. Then, we can define $dQ/dP =  g(\mathbf{x}_T)$. As a result, we obtain the following
\begin{align}
\mathbb{E}_{s,\mathbf{x}}^{Q} \bigl\{f(\mathbf{x}_t)\bigr\} -& f(\mathbf{x}_s) = \mathbb{E}_{s,\mathbf{x}}^{Q}  \Bigl \{ \int_s^t \Bigr [m_1(\tau, \mathbf{x}_{\tau}) \cdot D_{x^1} f(\mathbf{x}_{\tau})  \notag \\
&  + \sum\nolimits_{j=2}^n m_j(\tau, \mathbf{x}_{\tau}^{j-1}) \cdot D_{x^j} f(\mathbf{x}_{\tau}) +  \dfrac{1}{2} \operatorname{tr} \bigl(a(\tau, \mathbf{x}_{\tau}) D_{x^1}^2 f(\mathbf{x}_{\tau})\bigr) \notag \\
& \quad \quad + a(\tau, \mathbf{x}_{\tau}) D_{x^1} \log h(\tau, \mathbf{x}_{\tau}) \cdot D_{x^1}f(\mathbf{x}_{\tau})\Bigr] d\tau \Bigr \},   \label{Eq25}
\end{align}
since $\mathbf{x}_t \rightarrow \mathbf{x}_T$ a.e., we can let $t \rightarrow T$ and conclude by the Lebesque's dominance convergence theorem (see \cite[Chapter~4]{Roy88}).

In order to show \eqref{Eq20}, we only need to check the following condition
\begin{align}
\mathbb{E}_{s,\mathbf{x}}^{Q} \bigl\{f(\mathbf{x}_t)\bigr\} =  \int q^h(s, \mathbf{x}, t, \mathbf{y}) f(\mathbf{y}) d \mathbf{y}  \label{Eq26}
\end{align}
for any $f \in C_0^{\infty}(\mathbb{R}^{nd})$. That is,
\begin{align}
\mathbb{E}_{s,\mathbf{x}}^{Q} \bigl\{f(\mathbf{x}_t)\bigr\} &= \dfrac{1}{h(s, \mathbf{x}_s)} \mathbb{E}_{s,\mathbf{x}} \bigl\{ h(t, \mathbf{x}_t) f(\mathbf{x}_t)\bigr\} \notag \\
                                                                                           &= \dfrac{1}{h(s, \mathbf{x}_s)} \int q(s, \mathbf{x}, t, \mathbf{y}) h(t, \mathbf{y}) f(\mathbf{y}) d \mathbf{y} \notag \\
                                                                                           &= \int q^{h}(s, \mathbf{x}, t, \mathbf{y}) f(\mathbf{y}) d \mathbf{y}. \label{Eq27}
\end{align}
This completes the proof of Proposition~\ref{P1}.
\end{proof}

Next, we state the following proposition (without proof) which is a version of the result given in \cite{Jam74} (cf. Beurling \cite{Beu60}). Later, we use this proposition for proving Propositions~\ref{P4} and \ref{P5} in Section~\ref{S4}.

\begin{proposition} (cf. \cite[Theorems~2.1 and 3.1]{Jam74}) \label{P2}
Let $\mu_0$ and $\mu_T$ be two probability measures on $\mathbb{R}^{nd}$. Suppose that $q(s, \mathbf{x}, t, \mathbf{y})$, for $0 \le s < t \le T$ and $\mathbf{x}, \mathbf{y} \in \mathbb{R}^{nd}$, is a transition density. Then, there exists a unique pair of $\sigma$-finite measures $(\nu_0, \nu_T)$ on $\mathbb{R}^{nd}$ such that the measure $\mu$ on $\mathbb{R}^{nd} \times \mathbb{R}^{nd}$, which is defined by
\begin{align}
\mu(E) = \int_{E} q(0, \mathbf{x}, T, \mathbf{y}) \nu_0(d\mathbf{x})\nu_T(d\mathbf{y}),  \label{Eq28}
\end{align}
has marginals $\mu_0$ and $\mu_T$ (where $E$ is an arbitrary $\mathbb{R}^{nd} \times \mathbb{R}^{nd}$-Borel set). Furthermore, $\nu_0 \ll \mu_0$ and $\nu_T \ll \mu_T$ are mutually absolutely continuous measures.
\end{proposition}

Recall the following definition that will be useful later.
\begin{definition}
Assume that $\nu_2$ and $\nu_1$ are $\sigma$-finite measures defined in the same measure space. Then, the {\it relative entropy} of $\nu_2$ with respect to $\nu_1$ is defined by 
\begin{eqnarray}
 H(\nu_2 \vert \nu_1) = \left \{\begin{array}{l}
  \int \log \Bigl(\dfrac{d \nu_2}{d \nu_1} \Bigr) d \nu_2, \quad \quad \text{if} \quad  \nu_2 \ll \nu_1, \\
  +\infty \quad \quad \quad \quad\quad\quad \quad \text{otherwise}.
\end{array}\right.  \label{Eq29} 
\end{eqnarray}
\end{definition}

\section{Statement of the problem} \label{S3}
Let us consider the following controlled-diffusion process
\begin{align}
d \mathbf{x}_t^{u} = \Bigl( \mathbf{M} (t, \mathbf{x}_t^{u} ) + G u_t \Bigr) dt + G \sigma (t, \mathbf{x}_t^{u}) dW_t, \label{Eq30}
\end{align}
where $u_t$ is an admissible control that satisfies
\begin{enumerate} [(i)]
\item $u_t$ is an $\mathbb{R}^{d}$-valued process with measurable sample paths satisfying nonanticipatory condition, i.e., $(W_t -W_s)$ is independent of $u_r$, for $r \le s \le t$;
\item \eqref{Eq30} admits a weak solution in [0, T]; and
\item $\mathbb{E} \int_0^{T} \bigl \Vert u_t \bigr \Vert_{a^{-1}}^2 dt < +\infty$.
\end{enumerate}

Assume that we are given two probability measures $\mu_0$ and $\mu_T$, then we specifically consider the following problem (which was originally formulated by Schr\"{o}dinger in \cite{Sch31}, albeit in a slightly different context).

\begin{problem}\label{Pb1}  Find an optimal admissible control $u_t^{\ast}$ such that
\begin{enumerate} [(a)]
\item $\mathbf{x}_0^{u^{\ast}}$ and $\mathbf{x}_T^{u^{\ast}}$ are distributed according to $\mu_0$ and $\mu_T$, respectively; and 
\item $u_t^{\ast}$ (among all admissible controls satisfying ${\rm (i)}$-${\rm (iii)}$) minimizes the following cost functional
\begin{align}
J(u_t) = \mathbb{E} \int_0^{T} \dfrac{1}{2}\bigl \Vert u_t \bigr \Vert_{a^{-1}}^2 dt. \label{Eq31}
\end{align}
\end{enumerate}
\end{problem} 

In the following section (cf. Propositions~\ref{P4} and \ref{P5}), we provide a sufficient condition on the existence for such an optimal admissible control (i.e., a minimum energy control) for the above problem.

\section{Main results} \label{S4}
In this section, we present our main results -- where we first characterize the set of attainable distributions with respect to the class of admissible controls mention above in Section~\ref{S3}. Then, we provide a condition on the existence of an optimal admissible control forcing the controlled-diffusion process to the desired final attainable distribution starting from the initial distribution and, at the same time, we make connections to the problem of minimizing relative entropy subject to these end-point distribution constraints.\footnote{Such an optimal admissible control has been studied using the stochastic control arguments (i.e., the logarithmic transformations approach) from Fleming (e.g., see Fleming \cite{Flem78}; cf. Section~\ref{S2}).}

\subsection{The set of attainable distributions} \label{S4.1} 
Here, we provide a result that characterizes the set of attainable distributions for the controlled-diffusion process $\mathbf{x}_t^{u}$ in \eqref{Eq30} with respect to the above class of admissible controls. 

In what follows, we assume that $\mathbf{x}_t$ is a weak solution in $[0, T)$ to the SDE in \eqref{Eq2}, i.e., 
\begin{align*}
d \mathbf{x}_t &= \mathbf{M} (t, \mathbf{x}_t) dt + G \sigma (t, \mathbf{x}_t) dW_t, \quad \mathbf{x}_0 = \mathbf{\xi}, 
\end{align*}
where the initial point $\mathbf{x}_0 = \mathbf{\xi}$ is distributed according to $\mu_0$ and satisfies $\mathbb{E} \vert \mathbf{\xi} \vert^2 < +\infty$.

Recall that, the SDE in \eqref{Eq18}, with $h(t, \mathbf{x}) \in C_b^{1, 2} \bigl([0, T] \times \mathbb{R}^{nd}\bigr)$ satisfying \, ${\partial} h_{\varepsilon}(t, \mathbf{x})/{\partial t} + \mathcal{L}_{t,\mathbf{x}} h_{\varepsilon}(t, \mathbf{x}) = 0$ in $[0, T) \times \mathbb{R}^{nd}$, admits a weak solution $\mathbf{x}_t^{h}$ in $[0, T)$. Further, let $\mathcal{T}_{s,t}$, $0 \le s \le t \le T$, denote the transition semigroup for $\mathbf{x}_{\cdot}^{h}$. Note that the extended infinitesimal generator associated with $\mathbf{x}_t^{h}$ is given by
\begin{align}
 \mathcal{L}_{t,\mathbf{x}}^{h} = \mathcal{L}_{t,\mathbf{x}} + a(t, \mathbf{x}) D_{x^1} \log h(t, \mathbf{x}) \cdot D_{x^1},  \label{Eq32}
\end{align}
where $\mathcal{L}_{t,\mathbf{x}}$ is the infinitesimal generator associated with $\mathbf{x}_t$ (cf. equations~\eqref{Eq3} and \eqref{Eq24}).

Next, we assume that $f \colon \mathbb{R}^{nd} \rightarrow \mathbb{R}$ has continuous partial derivatives up to the second-order which, along with $f$, vanish at infinity and satisfies
\begin{align}
 \frac{\partial}{\partial s} \mathcal{T}_{s,t} f(\cdot) + \mathcal{L}_{s,\mathbf{x}}^{h} \mathcal{T}_{s,t} f(\cdot) = 0, \quad 0 \le s \le t \le T.  \label{Eq33}
\end{align}

Then, we have the following result that characterizes the set of attainable distributions for the controlled-diffusion process associated with \eqref{Eq30}.

\begin{proposition} \label{P3}
Given any admissible control $(u_{t})_{0 \le t \le T-\varepsilon}$, then the attainable distributions associated with $\mathbf{x}_t^{u}$ (cf. equation~\eqref{Eq30}) and that of $\mathbf{x}_t^{h}$ (cf. equation~\eqref{Eq18}), for each $t \in [0,T]$, are identical.
\end{proposition}

\begin{proof}
Suppose that $f$ satisfies \eqref{Eq33}. If we applying the It\^{o}'s formula to $\mathcal{T}_{s,t} f(\mathbf{x}_t^{u})$ for $ 0 \le s \le t \le T$, then we obtain
\begin{align}
f(\mathbf{x}_t^{u}) - \mathcal{T}_{0,t} f(\mathbf{x}_0^{u}) &= \int_0^t \biggl \{ \frac{\partial}{\partial s} \mathcal{T}_{s,t} f(\mathbf{x}_s^{u}) + m_1(s, \mathbf{x}_s^{u}) \cdot D_{x^1} f(\mathbf{x}_s^{u}) \notag \\
& \quad + \sum\nolimits_{j=2}^n m_j(s, {\mathbf{x}_s^{u,}}^{j-1}) \cdot D_{x^j} f(\mathbf{x}_s^{u}) + a(s, \mathbf{x}_s) u_s \cdot D_{x^1} f(\mathbf{x}_s^{u})  \notag \\
 & \quad + \dfrac{1}{2} \int_0^t \operatorname{tr} \bigl(a(s, \mathbf{x}_s^{u}) D_{x^1}^2 \mathcal{T}_{s,t} f(\mathbf{x}_s^{u})\bigr) \biggr\}ds \notag \\
 & \quad \quad + \int_0^t \sigma(s, \mathbf{x}_s^u) D_{x^1} \mathcal{T}_{s,t} f(\mathbf{x}_s^{u}) \cdot d W_s   \label{Eq34}
\end{align}
Note that \eqref{Eq33} holds a.e. with respect to the Lebesque measure and that of Assumption~\ref{AS1} (cf. Remark~\ref{R2}) ensures that the distribution law of $\mathbf{x}_t^u$, for each $t$, is absolutely continuous with respect to the Lebesque measure. 

Then, using \eqref{Eq33}, the first integral on the right-hand side of \eqref{Eq34} a.s. equals to
\begin{align}
\Bigl(a(s, \mathbf{x}_s^{u}) u_s - a(s, \mathbf{x}_s^{u}) D_{x^1} \log h(s, \mathbf{x}_s^{u}) \Bigr) \cdot D_{x^1} \mathcal{T}_{s,t} f(\mathbf{x}_s^{u}),    \label{Eq35}
\end{align}
which is integrable with respect to the underlying probability measure. Thus, if we take the expectations in the above equation (and noting that our choice of $h$ which gives an admissible Markov-type control $a(t, \mathbf{x}_t) D_{x^1} \log h(t, \mathbf{x}_t)$ (cf. Bene\v{s} in \cite{Ben70} for related discussions)), then we obtain
\begin{align}
\mathbb{E} \Bigl \{f(\mathbf{x}_t^{u})\Bigr\} - \mathbb{E} \Bigl \{f(\mathbf{x}_t^{h}) \Bigr\} &= \mathbb{E} \left \{ \int_0^T a(s, \mathbf{x}_s^{u}) u_s \cdot D_{x^1} \mathcal{T}_{s,t} f(\mathbf{x}_s^{u}) ds \right \} \notag \\
& \quad - \mathbb{E} \left \{ \int_0^T a(s, \mathbf{x}_s^{h}) D_{x^1} \log h(s, \mathbf{x}_s^{h}) \cdot D_{x^1} \mathcal{T}_{s,t} f(\mathbf{x}_s^{h}) ds \right\} \notag\\
&= 0.   \label{Eq36}
\end{align}
The claim follows easily from this, which completes the proof of Proposition~\ref{P3}.
\end{proof}

\subsection{Connection with stochastic control problems} \label{S4.2}
Here, we provide a sufficient condition on the existence for the optimal admissible control associated with Problem~\ref{Pb1}. Let $\mathcal{S}_{t}$ be an operator, acting on the set of $\sigma$-finite measures on $\mathbb{R}^{nd}$, defined as follow
\begin{align}
 \dfrac{d \mathcal{S}_t \mu}{d \lambda}(\mathbf{x}_t) = \int q(0, \mathbf{y}, t, \mathbf{x}) \mu(d \mathbf{y}),  \label{Eq37}
\end{align}
where ${d \mathcal{S}_t \mu}/{d \lambda}$ is the Radon-Nikodym derivative with respect to the Lebesque measure $\lambda$ and $q(s, \mathbf{y}, t, \mathbf{x})$ is the transition density associated with the SDE in \eqref{Eq2}.

First, let us consider Problem~\ref{Pb1} with a deterministic initial condition, i.e., when $\mu_0$ assumes a Dirac measure that is concentrated at a point $\mathbf{\xi} \in \mathbb{R}^{nd}$. Then, we have the following result.

\begin{proposition} \label{P4}
Suppose that $\mu_0$ is a Dirac measure which is concentrated at a point $\mathbf{\xi} \in \mathbb{R}^{nd}$.  Further, assume that
\begin{align}
 H(\mu_T \vert \mathcal{S}_T\mu_0) < + \infty \label{Eq38}
\end{align}
and let $h(t,\mathbf{x})$ be given by
\begin{align}
 h(t,\mathbf{x}) = \int q(t, \mathbf{x}, T, \mathbf{z}) \dfrac{d \mu_T}{d \mathcal{S}_T\mu_0} (\mathbf{z}) d \mathbf{z}, \quad (t, x) \in [0, T] \times \mathbb{R}^{nd}.  \label{Eq39}
\end{align}
Then, $u_t^{\ast} = a(t, \mathbf{x}_t) D_{x^1} \log h(t, \mathbf{x}_t)$ solves Problem~\ref{Pb1} with an optimal value of 
\begin{align*}
 J(u_t^{\ast}) = H(\mu_T \vert \mathcal{S}_T\mu_0).
 \end{align*}
\end{proposition}

\begin{proof}
Note that
\begin{align}
 h(0,\mathbf{x}_0) &= \int q(0, \mathbf{x}_0, T, \mathbf{z}) \dfrac{d \mu_T}{d \mathcal{S}_T\mu_0} (\mathbf{z}) d \mathbf{z} \notag\\
                             &= 1. \label{Eq40}
\end{align}
Recall that $h(t, \mathbf{x})$ belongs to $C_b^{1, 2} \bigl([0, T] \times \mathbb{R}^{nd}\bigr)$ and satisfies ${\partial} h_{\varepsilon}(t, \mathbf{x})/{\partial t} + \mathcal{L}_{t,\mathbf{x}}  h_{\varepsilon}(t, \mathbf{x}) = 0$ in $[0, T) \times \mathbb{R}^{nd}$ and 
\begin{align*}
h(T,\mathbf{x}) = \dfrac{d \mu_T} {d \mathcal{S}_T\mu_0} (\mathbf{x}), \quad \mathbf{x} \in \mathbb{R}^{nd}.
\end{align*}
Note that $h(t,\mathbf{x}_t)$ is martingale and $\mathbb{E} \bigl\{h(t,\mathbf{x}_t) \} =1$. Further, from Proposition~\ref{P1}, the SDE in \eqref{Eq18} admits a weak solution $\mathbf{x}_t^{h}$ in $[0, T]$ and then, by \eqref{Eq20}, $\mathbf{x}_T^{h}$ is distributed according to $\mu_T$.

Let $P_{\mathbf{x}}$ and $P_{\mathbf{x}^h}$ be measures induced by $\mathbf{x}_t$ and $\mathbf{x}_t^h$, respectively, on the path-space $C([0, T]; \mathbb{R}^{nd})$. Next, let us introduce the following change of measures
\begin{align}
\dfrac{d P_{\mathbf{x}^h}}{d P_{\mathbf{x}}} ({\mathbf{\xi}_{\cdot}}) = h(T, \mathbf{\xi}_T).  \label{Eq41}
\end{align}
Then, we have the following
\begin{align}
 \mathbb{E} \bigl\{ \log h(t,\mathbf{x}_t^h) \} &=  \mathbb{E} \bigl\{ h(t,\mathbf{x}_t) \log h(t,\mathbf{x}_t) \bigr\} \notag \\
                                                                      &\le  \mathbb{E} \bigl\{ h(T,\mathbf{x}_T) \log h(T,\mathbf{x}_T) \bigr\} \notag \\
                                                                     &= H(\mu_T \vert \mathcal{S}_T\mu_0), \label{Eq42}
\end{align}
where we used the fact that $h(t,\mathbf{x}_t) \log h(t,\mathbf{x}_t)$ is a submartingale process.\footnote{Note that $\phi(x) = x \log x$ is convex and bounded from below.}

Next, let us introduce the following sequence of stoping times
\begin{align*}
 \tau_n = \inf \bigl\{ s \, \bigl\vert \,\vert \mathbf{x}_s\vert > n \bigr\}
\end{align*}
and
\begin{align*}
 \tau_n (\omega) = T \quad \text{if} \quad \vert \mathbf{x}_{s}(\omega) \vert \le n \} \quad \text{for every} \quad 0 \le t \le T.
\end{align*}
For $t \le T$, if we apply Krylov's extension of the It\^{o} formula (cf. \cite[section~10, pp.~121--128]{Kry80}), then we have the following
\begin{align}
 \mathbb{E} \bigl\{ h(t \wedge \tau_n, \mathbf{x}_{t \wedge \tau_n}) \log h(t \wedge \tau_n, \mathbf{x}_{t \wedge \tau_n}) \bigr\} \le \mathbb{E} \int_0^{t \wedge \tau_n} \dfrac{1}{2} \bigl\Vert u_s^{\ast} \bigr\Vert_{a^{-1}}^2 h(s,\mathbf{x}_s) d s. \label{Eq43}
\end{align}
Further, from the optional sampling theorem, we have the following
\begin{align}
 \mathbb{E} \bigl\{ h(t \wedge \tau_n, \mathbf{x}_{t \wedge \tau_n}) \log h(t \wedge \tau_n, \mathbf{x}_{t \wedge \tau_n}) \bigr\} \le \mathbb{E} \bigl\{ h(t,\mathbf{x}_t) \log h(t,\mathbf{x}_t) \bigr\} \label{Eq44}
\end{align}
and, by \eqref{Eq42} 
\begin{align*}
 \mathbb{E} \bigl\{ h(t,\mathbf{x}_t) \log h(t,\mathbf{x}_t) \bigr\} < +\infty.
\end{align*}
On the other hand, $t \wedge \tau_n \rightarrow t$ as $n \rightarrow \infty$, then, from \eqref{Eq43} and \eqref{Eq44}, we have
\begin{align}
\mathbb{E} \int_0^{t}\dfrac{1}{2} \bigl\Vert u_s^{\ast} \bigr\Vert_{a^{-1}}^2 h(s,\mathbf{x}_s) d s \le \mathbb{E} \bigl\{ h(t, \mathbf{x}_t) \log h(t, \mathbf{x}_t) \bigr\}. \label{Eq45}
\end{align}
Moreover, if we apply the Fatou's lemma to the left-hand side of \eqref{Eq43} (which gives us the opposite inequality), then we have
\begin{align}
\mathbb{E} \int_0^{t}\dfrac{1}{2} \bigl\Vert u_s^{\ast} \bigr\Vert_{a^{-1}}^2 h(s,\mathbf{x}_s) d s = \mathbb{E} \bigl\{ h(t, \mathbf{x}_t) \log h(t, \mathbf{x}_t) \bigr\}.\label{Eq46}
\end{align}
Note that $h(t,\mathbf{x}_t) \log h(t,\mathbf{x}_t)$ is a submartingale process. Then, using again the Fatou's lemma and taking the limit $t \rightarrow T$, we have
\begin{align}
H(\mu_T \vert \mathcal{S}_T\mu_0) &= \mathbb{E} \bigl\{ h(T,\mathbf{x}_T) \log h(T,\mathbf{x}_T) \bigr\} \notag \\
                                                          &= \mathbb{E} \int_0^{T}\dfrac{1}{2} \bigl\Vert u_s^{\ast} \bigr\Vert_{a^{-1}}^2 h(s,\mathbf{x}_s) d s \notag \\
                                                          &= \mathbb{E} \int_0^{T}\dfrac{1}{2} \bigl\Vert u_s^{\ast} \bigr\Vert_{a^{-1}}^2 d s.  \label{Eq47}
\end{align}
If $u_t$ is any admissible control, then, using Girsanov's transformation (e.g., see \cite{Gir60}, \cite{DurB78} or \cite{TakW80}) and noting the fact that $\mathbf{x}_T^u$ is distributed according to $\mu_T$, we have the following
\begin{align}
 1 &=  \mathbb{E} \bigl\{ h(T,\mathbf{x}_T) \log h(T,\mathbf{x}_T) \bigr\} \notag \\
   &= \mathbb{E} \left\{ h(T,\mathbf{x}_T) \exp \left(\int_0^T \sigma^{-1} (t, \mathbf{x}_t^u) u_t \cdot dW_t - \int_0^{T} \dfrac{1}{2} \bigl\Vert u_t^{\ast} \bigr\Vert_{a^{-1}}^2 d t \right) \right\} \notag \\
   &\ge \exp \left\{ \mathbb{E}\left(\log h(T,\mathbf{x}_T)  + \int_0^T \sigma^{-1} (t, \mathbf{x}_t^u) u_t \cdot dW_t - \int_0^{T} \dfrac{1}{2} \bigl\Vert u_t^{\ast} \bigr\Vert_{a^{-1}}^2 d t \right) \right\} \notag \\                                                                    
   &= \exp \left\{ H(\mu_T \vert \mathcal{S}_T\mu_0) -  \mathbb{E} \int_0^{T} \dfrac{1}{2} \bigl\Vert u_t^{\ast} \bigr\Vert_{a^{-1}}^2 d t \right\}. \label{Eq48}
\end{align}
Hence, the above inequality further implies the following
\begin{align}
H(\mu_T \vert \mathcal{S}_T\mu_0)  \le \mathbb{E} \int_0^{T} \dfrac{1}{2} \bigl\Vert u_t^{\ast} \bigr\Vert_{a^{-1}}^2 d t. \label{Eq49}
\end{align}
This completes the proof of Proposition~\ref{P4}.
\end{proof}

\begin{remark} \label{R4}
Note that if $P_{\mathbf{x}_t}$ and $P_{\mathbf{x}_t^{u}}$ are measures induced by $\mathbf{x}_t$ and $\mathbf{x}_t^{u}$ on the path-space $C([0, T], \mathbb{R}^{nd})$. Then, using using Girsanov transformation (e.g., see \cite{Gir60}), we can reinterpret the cost functional $J(u_t)$ in terms of the {\it relative entropy} between $P_{\mathbf{x}_t^{u}}$ and $P_{\mathbf{x}_t}$, i.e., 
\begin{align}
  H(P_{\mathbf{x}_t^{u}} \vert P_{\mathbf{x}_t}) &= \int \log \dfrac{d P_{\mathbf{x}_t^{u}}}{d P_{\mathbf{x}_t}} d P_{\mathbf{x}_t^{u}} \notag \\
                                                                    &=  -\mathbb{E} \left\{ \int_0^{T} \sigma^{-1}(t, \mathbf{x}_t^{u}) u_t  \cdot d W_t  - \int_0^{T} \frac{1}{2}\bigl \Vert u_t \bigr \Vert_{a^{-1}}^2 d t \right\}  \notag \\
                                                                    &\equiv J(u_t).  \label{Eq50}
\end{align}
Moreover, when the admissible control is optimal (i.e., $u_t^{\ast} = a(t, \mathbf{x}_t) D_{x^1} \log h(t, \mathbf{x}_t)$), we have the following
\begin{align*}
 H(P_{\mathbf{x}_t^{u^{\ast}}} \vert P_{\mathbf{x}_t}) &= H(\mu_T \vert \mathcal{S}_T\mu_0) \\
                                                                                     & \equiv J(u_t^{\ast}),
\end{align*}
which implies the global {\it relative entropy} is exactly equal to the {\it relative entropy} between the final measures $\mu_T$ and $\mathcal{S}_T\mu_0$.
\end{remark}

Note that, from Proposition~{\ref{P2}}, for $\mu_0$ and $\mu_T$ (with $\mu_T \ll \mathcal{S}_T \mu_0$), there exist two $\sigma$-finite measures $\nu_0$ and $\nu_T$ such that the statement in \eqref{Eq28} holds. Letting $\rho_T(\mathbf{x}) = d\nu_T/d\nu_0$, then we have the following relations
\begin{align}
\dfrac{d \mu_T}{d \lambda} & = \rho_T(\mathbf{x}) \int q(0, \mathbf{y}, T, \mathbf{x}) \nu_0(d\mathbf{y}) \label{Eq51}
\end{align}
and
\begin{align}
\dfrac{d \mu_0}{d \nu_0} = \int q(0, \mathbf{x}, T, \mathbf{z}) \rho_T(\mathbf{z}) d\mathbf{z}.  \label{Eq52}
\end{align}

For any initial random variable $\mathbf{x}_0 = \mathbf{\xi}$ distributed according to $\mu_0$ and satisfying $\int \vert \mathbf{\xi}\vert^2 d \mu_0< +\infty$, then we have the following result which is a generalization of Proposition~\ref{P4}.

\begin{proposition} \label{P5}
Suppose that $H(\mu_T \vert \mathcal{S}_T\nu_0) < +\infty$ and $\int \bigl({d \mu_0}/{d \nu_0}\bigr)d \mu_0 < +\infty$. Let $h(t,\mathbf{x})$ be given by
\begin{align}
 h(t,\mathbf{x}) = \int q(t, \mathbf{x}, T, \mathbf{z}) \rho_T(\mathbf{z}) d\mathbf{z}, \quad (t, x) \in [0, T] \times \mathbb{R}^{nd}. \label{Eq53}
\end{align}
 Then, $\mathbf{u}_t^{\ast} = a(t, \mathbf{x}_t) D_{x^1} \log h(t, \mathbf{x}_t)$ solves Problem~\ref{Pb1} with an optimal value of
\begin{align}
 J(u_t^{\ast}) &= \mathbb{E} \int_0^{T} \frac{1}{2}\bigl \Vert u_t^{\ast} \bigr \Vert_{a^{-1}}^2 d t \notag \\
                                    &= H(\mu_T \vert \mathcal{S}_T\nu_0) - H(\mu_0 \vert \nu_0) . \label{Eq54}
\end{align}
\end{proposition}

\begin{proof}
First, let us show that $h(t, \mathbf{x}) = \mathbb{E}_{t, \mathbf{x}} \bigl\{\rho_T(\mathbf{x}_T) \bigr\}$. Note that, from \eqref{Eq54}, this is true if we show $\mathbb{E}_{t, \mathbf{x}} \bigl\{\rho_T(\mathbf{x}_T) \bigr\} < +\infty$, i.e.,

\begin{align}
 \mathbb{E}_{t, \mathbf{x}} \bigl\{\rho_T(\mathbf{x}_T) \bigr\} &= \int \rho_T(\mathbf{x}) d\mathcal{S}_T\mu_0 \notag \\
                                                                                               &= \int \rho_T(\mathbf{x}) \left(\int q(0, \mathbf{y}, T, \mathbf{x})d \mu_0(\mathbf{y}) \right) d \mathbf{x} \notag\\
                                                                                               &= \int \left(\int q(0, \mathbf{y}, T, \mathbf{x}) \rho_T(\mathbf{x}) d \mathbf{x} \right) d \mu_0(\mathbf{y}) \notag\\
                                                                                               &= \int \frac{d \mu_0}{d \nu_0}d \mu_0 < +\infty.  \label{Eq55}
\end{align}
Then, we can process as follow (cf. the proof part of Proposition~\ref{P4}). Let us introduce the following change of measures
\begin{align}
\dfrac{d P_{\mathbf{x}^h}}{d P_{\mathbf{x}}} ({\mathbf{\xi}_{\cdot}}) = \dfrac{\rho_T(\mathbf{\xi}_T)}{h(0, \mathbf{\xi}_0)}.  \label{Eq56}
\end{align}
Then, we have the following (cf. equation~\eqref{Eq42})
\begin{align}
 \mathbb{E} \Bigl\{ \log h(t,\mathbf{x}_t^h) \Bigr\} &= \mathbb{E} \left \{ \dfrac{1}{h(0,\mathbf{x}_0)} \mathbb{E} \bigl\{ h(t,\mathbf{x}_t) \log h(t,\mathbf{x}_t) \bigl \vert \mathbf{x}_0 \bigr\} \right \}  \notag \\
                                                                      &\le \mathbb{E} \left \{\dfrac{\rho_T(\mathbf{x}_T)}{h(0,\mathbf{x}_0)} \log h(T,\mathbf{x}_T) \right\} \notag \\
                                                                      &\le \mathbb{E} \Bigl\{ \log h(T,\mathbf{x}_T^h) \Bigr\} \notag \\
                                                                     &= H(\mu_T \vert \mathcal{S}_T\nu_0) \label{Eq57}
\end{align}
and in place of \eqref{Eq43}, we have the following
\begin{align}
 \mathbb{E} \left \{ \dfrac{1}{h(0,\mathbf{x}_0)} \mathbb{E} \bigl\{ h(t \wedge \tau_n, \mathbf{x}_{t \wedge \tau_n}) \log h(t \wedge \tau_n, \mathbf{x}_{t \wedge \tau_n}) \bigl \vert \mathbf{x}_0 \bigr \} \right \} -  \mathbb{E} \Bigl\{ \log h(0, \mathbf{x}_0) \Bigr\} \notag \\
  = \mathbb{E} \left \{ \dfrac{1}{h(0,\mathbf{x}_0)} \mathbb{E} \int_0^{t \wedge \tau_n} \dfrac{1}{2} \bigl\Vert u_s^{\ast} \bigr\Vert_{a^{-1}}^2 h(s,\mathbf{x}_s) d s \right \}, \label{Eq58}
\end{align}
where
\begin{align*}
 \mathbb{E} \bigl\{ \log h(0, \mathbf{x}_0) \bigr\} &= \int \log \dfrac{d\mu_0} {d\nu_0} d \mu_0 \notag \\
                                                                             &= H(\mu_T \vert \nu_0) < +\infty
\end{align*}
and, from Jensen's inequality, we further have the following
\begin{align*}
 0 \le  \int \log \dfrac{d\mu_0} {d\nu_0} d \mu_0 \le \log \int \dfrac{d\mu_0} {d\nu_0} d \nu_0  < +\infty.
\end{align*}
Then, using the limit arguments (i.e., the Fatou's lemma) as in Proposition~\ref{P5}, we obtain the following
\begin{align}
 \mathbb{E} \int_0^{T} \dfrac{1}{2} \bigl\Vert u_t^{\ast} \bigr\Vert_{a^{-1}}^2 d t = H(\mu_T \vert \mathcal{S}_T\nu_0) - H(\mu_0 \vert \nu_0). \label{Eq59}
\end{align}
Moreover, for any admissible control $u_t$, then we have
\begin{align}
 \mathbb{E} \int_0^{T} \dfrac{1}{2} \bigl\Vert u_t \bigr\Vert_{a^{-1}}^2 d t \ge H(\mu_T \vert \mathcal{S}_T\nu_0) - H(\mu_0 \vert \nu_0). \label{Eq60}
\end{align}
This completes the proof of Proposition~\ref{P5}.
\end{proof}

Note that the conditions (i.e., $H(\mu_T \vert \mathcal{S}_T\nu_0) < +\infty$ and $\int \bigl( {d \mu_0}/{d \nu_0} \bigr)d \mu_0 < +\infty$) under which Proposition~\ref{P5} holds are rather difficult to verify. However, when $\mu_0$ has compact support, we can relax them with suitable conditions due to the following lemma.

\begin{lemma} \label{L1}
Suppose that $\mu_0$ has compact support and $H(\mu_T \vert \mathcal{S}_T\mu_0) < +\infty$. Then, we have
\begin{align*}
H(\mu_T \vert \mathcal{S}_T\nu_0) < +\infty \quad \text{and} \quad \int \Bigl(\dfrac{d \mu_0}{d \nu_0}\Bigr) d \mu_0 < +\infty.
\end{align*}
\end{lemma}

\begin{proof}
Note that $h(t, \mathbf{x})$ is smooth and if $\mu_0$ has compact support. Then, we have
\begin{align*}
\int \Bigl(\dfrac{d \mu_0}{d \nu_0}\Bigr) d \mu_0 = \int h(0, \mathbf{x}) d \mu_0(\mathbf{x}) < +\infty.
\end{align*}
Moreover,
\begin{align*}
H(\mu_T \vert \mathcal{S}_T \nu_0) & = \int \log \dfrac{d \mu_T}{d\mathcal{S}_T \mu_0} d \mu_T + \int \log \dfrac{d\mathcal{S}_T \mu_0} {d\mathcal{S}_T \nu_0} d\mu_T \\
                                                          & = H(\mu_T \vert \mathcal{S}_T \mu_0) + \int \log \dfrac{d\mathcal{S}_T \mu_0} {d\mathcal{S}_T \nu_0} d\mu_T.
\end{align*}
Then, we have to show that $\int \log (d\mathcal{S}_T \mu_0/d\mathcal{S}_T \nu_0) d\mu_T$ exists.

Let us define $\log^{-} \phi = \max \bigl\{ -\log \phi,\, 0 \bigr\}$. Then, from Jensen inequality, we have the following
\begin{align*}
\log^{-} \dfrac {d\mathcal{S}_T \mu_0} {d \mu_T} (\mathbf{y}) &=  \log^{-} \int \dfrac{d \mu_0} {d \nu_0}(\mathbf{x})  \dfrac{q(0, \mathbf{x}, T, \mathbf{y}) } {(d\mathcal{S}_T \nu_0/d \lambda)(\mathbf{y})} d\nu_0 (\mathbf{x}) \\
                                                          &\le  \int \left ( \log^{-} \dfrac{d \mu_0} {d \nu_0}(\mathbf{x}) \right)  \dfrac{q(0, \mathbf{x}, T, \mathbf{y}) } {(d\mathcal{S}_T \nu_0/d \lambda)(\mathbf{y})} d\nu_0 (\mathbf{x}).
\end{align*}
Hence, we have
\begin{align*}
\int \log^{-} \dfrac {d\mathcal{S}_T \mu_0} {d \nu_0} (\mathbf{y}) d \mu_T (\mathbf{y}) &\le \int \left(\int \left ( \log^{-} \dfrac{d \mu_0} {d \nu_0}(\mathbf{x}) \right) q(0, \mathbf{x}, T, \mathbf{y}) d\nu_0 (\mathbf{x}) \right) d\nu_T (\mathbf{y}) \\
&= \int \log^{-} \dfrac{d \mu_0} {d \nu_0} d\mu_0 \\
&= \int \log^{-} h(0, \mathbf{x}) d\mu_0 (\mathbf{x}) < +\infty.
\end{align*}
Moreover, we have the following
\begin{align*}
\int \log \dfrac {d\mathcal{S}_T \mu_0} {d \mathcal{S}_T \nu_0} d \mu_T &\le \int \log \dfrac {d\mathcal{S}_T \mu_0} {d \mathcal{S}_T \nu_0} d \mu_T \\
&=  \log \int h(T, \mathbf{y}) d \mathcal{S}_T\mu_0 (\mathbf{y}) \\
&= \int \log h(0, \mathbf{x}) d\mu_0 (\mathbf{x}) < +\infty.
\end{align*}
This completes the proof of Lemma~\ref{L1}.
\end{proof}

\begin{remark} \label{R5}
In Problem~\ref{Pb1}, we can also include a state dependent term in the cost functional of \eqref{Eq31}, i.e.,
\begin{align*}
 J(\mathbf{x}_t^u, u_t) = \mathbb{E} \int_0^{T} \left \{ \frac{1}{2}\bigl \Vert u_t \bigr \Vert_{a^{-1}}^2 + \kappa(\mathbf{x}_t^{u})\right \} d t,
\end{align*}
where $\kappa$ is a nonnegative, real-valued continuous function on $\mathbb{R}^{nd}$. 

Then, we can proceed in the same way as above if we take $h(t, \mathbf{x})$ in the kernel of the operator $\bigl(\partial / \partial t + \mathcal{L}_{t,\mathbf{x}} - \kappa\bigr)$ and $\tilde{q}(s, \mathbf{x}, t, \mathbf{y})$ as the fundamental solution of $\partial \tilde{q}(t, \mathbf{x}, T, \mathbf{y})/\partial t + \mathcal{L}_{t,\mathbf{x}} \tilde{q}(t, \mathbf{x}, T, \mathbf{y}) - \kappa(t, \mathbf{x}) \tilde{q}(t, \mathbf{x}, T, \mathbf{y}) = 0$ in $[0, T) \times \mathbb{R}^{nd}$ and $\lim_{t \uparrow T} \tilde{q}(t, \mathbf{x}, T, \mathbf{y}) = \delta_{\mathbf{y}}(\mathbf{x})$ for $\mathbf{x}, \mathbf{y} \in \mathbb{R}^{nd}$. Moreover, $h(t, \mathbf{x})$ admits the following probabilistic representation
\begin{align*}
h(s, \mathbf{x}) = \mathbb{E}_{s, \mathbf{x}} \left \{h(T, \mathbf{x}_T) \exp \left \{ -\int_s^{T} \kappa(\mathbf{x}_t) dt \right \} \right \} d t.
\end{align*}
Note that, $\tilde{q}(s, \mathbf{x}, t, \mathbf{y})$ is the transition density of the killed diffusion process with the same drift and diffusion terms as that of \eqref{Eq2} with killing rate (or potential) $\kappa$ (e.g., see \cite{BluG68} for additional discussions).
\end{remark}

\section{Remarks on the invariance property of the path-space measure} \label{S5}
In this section, we briefly remark on the invariance property of the path-space measure of the diffusion process pertaining to the chain of distributed systems. Note that such an interpretation makes sense if the diffusion process $(\mathbf{x}_{t})_{0 \le t \le T}$, which is associated with the SDE in \eqref{Eq2}, is considered as a random variable with values on a space of functions $C([0,T]; \mathbb{R}^{nd})$ containing its trajectories (e.g., see \cite{Wak89}, \cite{DurB78} or \cite{TakW80}). As a result, we can determine {\it local information} about the measure induced by $\mathbf{x}_{[0,T]}$. For example, for a given $\mathbf{\varphi} \in C^2([0,T]; \mathbb{R}^{nd})$ and small $\varepsilon > 0$, we can provide an asymptotic estimate on the probability of a small $\varepsilon$-tube around $C^2([0,T]; \mathbb{R}^{nd})$-function using
\begin{align}
 \mathbb{P} \left \{ \Vert\mathbf{x}_{\cdot} - \mathbf{\varphi} \Vert < \varepsilon \right\} \sim \kappa_{\varepsilon} \exp \Bigl \{ -\int_0^T L(t, \mathbf{\varphi}, \dot{\mathbf{\varphi}})dt \Bigr\} \quad \text{as} \quad \varepsilon \rightarrow 0, \label{Eq61}
\end{align}
where $L(t, \varphi, \dot{\mathbf{\varphi}})$ is the Lagrange function given by\footnote{$\bigl \Vert \mathbf{M}(t, \mathbf{\varphi}) - \dot{\mathbf{\varphi}} \bigr \Vert_{\tilde{a}^{-1}}^2 \triangleq \bigl \Vert \sigma^{-1} G^T (t, \mathbf{\varphi}) \bigl(\mathbf{M}(t, \mathbf{\varphi}) - \dot{\mathbf{\varphi}}\bigr) \bigr \Vert^2$.}
\begin{align}
 L(t, \mathbf{\varphi}, \dot{\mathbf{\varphi}}) = \frac{1}{2} \Bigl\Vert \mathbf{M}(t, \mathbf{\varphi}) - \dot{\mathbf{\varphi}} \Bigr \Vert_{\tilde{a}^{-1}}^2. \label{Eq62}
\end{align}
Note that the above asymptotic estimate in \eqref{Eq61} provides a probabilistic interpretation for the most probable paths, i.e., the most probably trajectories that minimize the functional $\int_0^T L(t, \varphi, \dot{\mathbf{\varphi}}) dt$. Moreover, these extreme trajectories (which belong to $C^2([0,T]; \mathbb{R}^{nd})$) are solutions to the following Euler-Lagrange differential equation
\begin{align}
\dfrac {\partial}{\partial \mathbf{\varphi}} L(t, \varphi, \dot{\mathbf{\varphi}}) - \dfrac{d} {d t} \dfrac {\partial} {\partial \dot{\mathbf{\varphi}}} L(t, \mathbf{\varphi}, \dot{\mathbf{\varphi}}) = 0. \label{Eq63}
\end{align}
The following result shows that adding a perturbation $G\, a(t, \mathbf{x}) D_{x^1} \log h(t, \mathbf{x})$ to the original drift term $\mathbf{M}(t, \mathbf{x})$ does not change the extreme trajectories of the diffusion process associated with the chain of distributed systems in \eqref{Eq2}.

\begin{proposition} \label{P6}
Assume that $\mathbf{M}(t, \mathbf{x}_t) \in C_b^2([0,T] \times \mathbb{R}^{nd}; \mathbb{R}^{nd})$ and $\sigma(t, \mathbf{x}_t) \in C_b^2([0,T] \times \mathbb{R}^{nd};  \mathbb{R}^{d \times d})$; and suppose that $h(t, \mathbf{x}) \in C_b^{1, 2} \bigl([0, T] \times \mathbb{R}^{nd}\bigr)$ is a strictly positive function that satisfies $\partial h(t, \mathbf{x})/ \partial t + \mathcal{L}_{t,\mathbf{x}} h(t, \mathbf{x}) = 0$ in $[0, T) \times \mathbb{R}^{nd}$. Then, the diffusion processes $\hat{\mathbf{x}}_t$ and $\tilde{\mathbf{x}}_t$ with the same diffusion term $\sigma(t, \mathbf{x})$ and whose drifts are $\mathbf{M}(t, \mathbf{x})$ and $\mathbf{M}(t, \mathbf{x}) + G\,a(t, \mathbf{x}) D_{x^1} \log h(t, \mathbf{x})$, respectively, have the same extreme trajectories.
\end{proposition}

\begin{proof}
Note that we can rewrite the Euler-Lagrange differential equation associated with $\hat{\mathbf{x}}_t$ as follow
\begin{align}
\dfrac {\partial}{\partial \varphi} L - \dfrac{d} {d t} \dfrac {\partial} {\partial \dot{\varphi}} L &= \sigma^{-1}G^T\Bigl(\mathbf{M} - \dot{\varphi}\Bigr) D_{\varphi} \sigma^{-1} G^T\Bigl(\mathbf{M} - \dot{\varphi}\Bigr) \notag \\
&\quad + \sigma^{-1}G^T\Bigl(\mathbf{M} - \dot{\varphi}\Bigr) \sigma^{-1}G^T D_{\varphi} \mathbf{M} +  \dfrac{\partial \sigma^{-1}} {\partial t} G^T\Bigl(\mathbf{M} - \dot{\varphi}\Bigr) \sigma^{-1}G^T \notag \\
& \quad  + \sigma^{-1}G^T\Bigl(\dfrac{\partial \mathbf{M}}{\partial t} - \ddot{\varphi} \Bigr) \sigma^{-1}G^T  + \sigma^{-1} G^T\Bigl(\mathbf{M} - \dot{\varphi}\Bigr) \dfrac{\partial \sigma^{-1}} {\partial t} G^T\notag \\
&= 0. \label{Eq64}
\end{align}
where $D_{{\varphi}}$ denotes the vector derivative with respect to $\varphi$.

Similarly, for the Lagrange function $L^h$ associated with $\tilde{\mathbf{x}}_t$, i.e.,
\begin{align}
 L^h(t, \mathbf{\varphi}, \dot{\mathbf{\varphi}}) = \frac{1}{2} \Bigl \Vert \mathbf{M}(t, \mathbf{\varphi}) + G\,a(t, \mathbf{\varphi}) G^T D_{\varphi} \log h(t, \mathbf{\varphi})- \dot{\mathbf{\varphi}} \Bigr \Vert_{\tilde{a}^{-1}}^2, \label{Eq65}
\end{align}
we can further compute the associated Euler-Lagrange differential equation as follow
\begin{align}
\dfrac {\partial}{\partial \varphi} L^h - \dfrac{d} {d t} \dfrac {\partial} {\partial \dot{\varphi}} L^h & =  \sigma^{-1}G^T\Bigl(\mathbf{M} - \dot{\varphi}\Bigr) D_{\varphi} \sigma^{-1} G^T\Bigl(\mathbf{M} - \dot{\varphi}\Bigr) \notag \\
& + \sigma^{-1}G^T\Bigl(\mathbf{M} - \dot{\varphi}\Bigr) \sigma^{-1}G^T D_{\varphi} \mathbf{M} +  \dfrac{\partial \sigma^{-1}} {\partial t} G^T\Bigl(\mathbf{M} - \dot{\varphi}\Bigr) \sigma^{-1}G^T \notag \\
&   + \sigma^{-1}G^T\Bigl(\dfrac{\partial \mathbf{M}}{\partial t} - \ddot{\varphi} \Bigr) \sigma^{-1}G^T  + \sigma^{-1} G^T\Bigl(\mathbf{M} - \dot{\varphi}\Bigr) \dfrac{\partial \sigma^{-1}} {\partial t} G^T\notag \\
& + \sigma G^T D_{\varphi} \log h D_{\varphi} \sigma^{-1} G^T \Bigl(\mathbf{M} - \dot{\varphi} \Bigr) + \sigma G^T D_{\varphi} \log h\, \sigma^{-1} G^T D_{\varphi} \mathbf{M}  \notag \\
& + \Bigl(\sigma^{-1} G^T \Bigl(\mathbf{M} - \dot{\varphi} \Bigr) + \sigma G^T D_{\varphi} \log h \Bigr) D_{\varphi} \sigma G^T D_{\varphi} \log h \notag \\
& + \Bigl(\sigma^{-1} G^T \Bigl(\mathbf{M} - \dot{\varphi} \Bigr) + \sigma G^T D_{\varphi} \log h \Bigr) \sigma G^T D_{\varphi}^2 \log h \notag \\
& + \dfrac{ \partial \sigma} {\partial t} G^T D_{\varphi} \log h  + \sigma G^T \dfrac{\partial D_{\varphi} \log h} {\partial t} \notag \\
&= 0. \label{Eq66}
\end{align}
Note that the first five terms in \eqref{Eq66} are identical to the Euler-Lagrange differential equation associated with $L$ (cf. equation~\eqref{Eq64}). Further, noting that the last six terms in \eqref{Eq66}, i.e.,
\begin{align*}
& \sigma G^T D_{\varphi} \log h D_{\varphi} \sigma^{-1} G^T \Bigl(\mathbf{M} - \dot{\varphi} \Bigr) + \sigma G^T D_{\varphi} \log h\, \sigma^{-1} G^T D_{\varphi} \mathbf{M}  \notag \\
 & +\sigma^{-1} G^T \Bigl(\mathbf{M} - \dot{\varphi} \Bigr) D_{\varphi} \sigma G^T D_{\varphi} \log h  + \Bigl( \sigma G^T D_{\varphi} \log h \Bigr) D_{\varphi} \sigma G^T D_{\varphi} \log h  \notag \\
 & + \sigma^{-1} G^T \Bigl(\mathbf{M} - \dot{\varphi} \Bigr)\sigma G^T D_{\varphi}^2 \log h + \Bigl(\sigma G^T D_{\varphi} \log h \Bigr) \sigma G^T D_{\varphi}^2 \log h \notag \\
 & + \dfrac{ \partial \sigma} {\partial t} G^T D_{\varphi} \log h  + \sigma G^T \dfrac{\partial D_{\varphi} \log h} {\partial t},
\end{align*}
with additional steps, reduced to the logarithmic transformation of $h$ that satisfies the following (cf. equation~\eqref{Eq10})
\begin{align}
 \frac{\partial \log h}{\partial t} + \mathbf{M} D_{\varphi} \log h + \dfrac{1}{2} \operatorname{tr}\Bigl(a G^T D_{\varphi}^2 \log h \Bigr)- \frac{1}{2} a \Bigr(G^T D_{\varphi} \log h \Bigl)^2= 0.  \label{Eq67}
\end{align}
Hence, we see that the above two Lagrangians (i.e., the Lagrange functions in \eqref{Eq62} and \eqref{Eq65}) yield the same Euler-Lagrange equation. The claim follows easily from this, which completes the proof of Proposition~\ref{P6}.
\end{proof}

\medskip
Received xxxx 20xx; revised xxxx 20xx.
\medskip

\end{document}